\documentclass[12pt]{amsart}

\usepackage{amsthm}
\usepackage{amsmath}
\usepackage{amssymb}
\usepackage{amssymb, amsfonts, amsbsy, latexsym, epsfig, color, verbatim}
\usepackage{enumerate}

\newtheorem{thm}{Theorem}[section]
\newtheorem{theorem}[thm]{Theorem}

\newtheorem{corollary}[thm]{Corollary}
\newtheorem{problem}[thm]{Problem}
\newtheorem{example}[thm]{Example}

\newtheorem{lemma}[thm]{Lemma}
\newtheorem{proposition}[thm]{Proposition}

\theoremstyle{remark}
\newtheorem{remark}[thm]{Remark}

\newcommand{\norm}[1]{\|#1\|}

\newcommand{\RR}{\mathbb R}

\newcommand{\CC}{\mathbb C}

\newcommand{\HH}{\mathbb H}

\newcommand{\cA}{\mathcal A}
\newcommand{\Tr}{\mbox{Tr}}
\newcommand{\Nul}{\mbox{Null}}

\newcommand{\inner}[2]{\langle #1,#2 \rangle}
\newcommand{\spann}{\mbox{\rm span}}
\newcommand{\abs}[1]{\left| #1 \right|}

\title{Phase retrieval by projections}

\author[Cahill, Casazza, Peterson, Woodland
 ]{Jameson Cahill, Peter G. Casazza, Jesse Peterson and Lindsey Woodland}
\address{Department of Mathematics, University
of Missouri, Columbia, MO 65211-4100}

\thanks{The first, second, and fourth authors were supported by
 NSF DMS 1008183;  NSF ATD 1042701;  AFOSR  DGE51; and FA9550-11-1-0245}

\email{Casazzap@missouri.edu; jdpq6c@mail.missouri.edu; lmwvh4@mail.missouri.edu}

\subjclass[2000]{}

\begin{document}

\begin{abstract}
The problem of recovering a vector from the absolute values of its inner products against a family of measurement vectors has been well studied in mathematics and engineering.  A generalization of this phase retrieval problem also exists in engineering: recovering a vector from measurements consisting of norms of its orthogonal projections onto a family of subspaces.  There exist semidefinite programming algorithms to solve this problem, but much remains unknown for this more general case.  Can families of subspaces for which such measurements are injective be completely classified?  What is the minimal number of subspaces required to have injectivity?  How closely does this problem compare to the usual phase retrieval problem with families of measurement vectors?  In this paper, we answer or make incremental steps toward these questions.  We provide several characterizations of subspaces which yield injective measurements, and through a concrete construction, we prove the surprising result that phase retrieval can be achieved with $2M-1$ projections of arbitrary rank
in $\HH_M$.
  Finally we present several open problems as we discuss issues unique to the phase retrieval problem with subspaces.
\end{abstract}

\maketitle

\section{Introduction}
In many engineering applications signals pass through linear systems, but in this process the recorded phase information can be lost or distorted.  Examples of this problem occur in speech recognition \cite{BR, RJ, PDH}, quantum state tomography \cite{RBSC}, and optics applications such as X-rays, crystallography, and electron microscopy \cite{BM, F78, F82}. {\it Phase retrieval} is the problem of recovering a signal from the absolute values of linear measurement coefficients called {\it intensity measurements}.  Note multiplying a signal by a global phase factor does not effect these coefficients, so we seek signal recovery mod a global phase factor.

There are two main approaches to this problem of phase retrieval.  One is to restrict the problem to a subclass of signals on which the intensity measurements become injective.  The other is to use a larger family of measurements so that the intensity measurements map any signal injectively.  The latter approach in phase retrieval first appears in \cite{BCE} where the authors examine injectivity of intensity measurements for finite Hilbert spaces.  The authors completely characterize measurement vectors in the real case which yield such injectivity, and they provide a surprisingly small upper bound on the minimal number of measurements required for the complex case.  This sparked an incredible volume of current phase retrieval research \cite{ABFM, BBCE, BCMN, CESV, CL, CSV, DH, EM} focused on algorithms and conditions guaranteeing injective and stable intensity measurements.

Given a signal $x$ in a Hilbert space, intensity measurements may also be thought of as norms of $x$ under rank one projections.  Here the spans of measurement vectors serve as the one dimensional range of the projections.  In some applications however, a signal must be reconstructed from the norms of higher dimensional components.  In X-ray crystallography for example, such a problem arises with crystal twinning \cite{D}.  In this scenario, there exists a similar phase retrieval problem: given subspaces $\{W_n\}_{n=1}^N$ of an $M$-dimensional Hilbert space $\HH_M$ and orthogonal projections $P_n:\HH_M\rightarrow W_n$, can we recover any $x\in \HH_M$ (up to a global phase factor) from the measurements $\{\norm{P_nx}\}_{n=1}^N$?  This problem was recently studied in \cite{BE} where the authors use semidefinite programming to develop a reconstruction algorithm for when the $\{W_n\}_{n=1}^N$ are equidimensional random subspaces.  Most results using random intensity measurements require the cardinality of measurements to scale linearly with the dimension of the signal space along with an additional logarithmic factor \cite{CSV}, but this logarithmic factor was recently removed in \cite{CL}.  Similarly, signal reconstruction from the norms of equidimensional random subspace components are possible with the cardinality of measurements scaling linearly with the dimension \cite{BE}.

In contrast to these results concerning the phase retrieval problem using subspaces, and much like \cite{BCE}, we seek to better characterize the subspaces $\{W_n\}_{n=1}^N$ of $\HH_M$ for which the measurements $\{\norm{P_nx}\}_{n=1}^N$ are injective for all $x\in \HH_M$.  To set notation, given subspaces $\{W_n\}_{n=1}^N$ of $\RR^M$ with orthogonal projections $P_n:\RR^M\rightarrow W_n$, we consider the measurements $\cA:\RR^M/\{\pm 1\}\rightarrow \RR^N$ given by
\begin{equation}\label{A}
	\cA(x)(n):=\norm{P_nx}^2.
\end{equation}
The phrase ``{\it $\{W_n\}_{n=1}^N$ allows phase retrieval}'' will be synonomous with $\cA$ being injective.  For the well studied case of $\dim W_n=1$ for all $n$, whether or not $\cA$ is injective shall be referred to as the {\it classical phase retrieval} problem.  As shown in the notation above, our primary focus will be $\RR^M$, but some results hold for the complex case as well.  When giving a result for $\CC$ instead of $\RR$, we will make this explicit.

While we provide several characterizations of subspaces $\{W_n\}_{n=1}^N$ for which $\cA$ is injective, several unique challenges remain, and this problem has significant potential for further study.  In this current paper, we present the following.  In section \ref{sec1} we will discuss classical phase retrieval and connect classical phase retrieval to phase retrieval with subspace components.  We also provide an upper bound on the minimal number of subspaces required for phase retrieval (in $\RR^M$ and $\CC^M$) by concretely constructing subspaces which allow phase retrieval.  In section \ref{sec2} we will provide several characterizations for subspaces $\{W_n\}_{n=1}^N$ which allow phase retrieval.  Section \ref{sec3} further discusses additional properties of these subspaces.  Finally section \ref{sec4} presents several open problems and incremental steps toward solving them.  This section also includes discussions of future work concerning the phase retrieval problem with subspace components.
For a background on frame theory we recommend \cite{CK}.

\section{Phase retrieval with subspace components and classical phase retrieval} \label{sec1}

As the phase retrieval problem of using norms of subspace components is a generalization of classical phase retrieval, several tools are pertinent to both problems.  In fact, our first approach to the phase retrieval problem from subspace components will be to reduce it to the classical case.  To this end, we discuss existing characterziations for classical phase retrieval.

Given a family of vectors $\Phi=\{\varphi_n\}_{n=1}^N$ in $\RR^M$, the {\it spark} of $\Phi$ is defined as the cardinality of the smallest linearly dependent subset of $\Phi$.  When $\mbox{spark}(\Phi)=M+1$ every subset of size $M$ is linearly independent, and $\Phi$ is said to be {\it full spark}.  Corollary 2.6 in \cite{BCE} shows when $N\geq 2M-1$, $\Phi$ allows phase retrieval when $\Phi$ is full spark.  This condition is not necessary, however, so \cite{BCE} also introduces a necessary and sufficient characterization called the {\it complement property} in modern terminology.  The family of vectors $\Phi=\{\varphi_n\}_{n=1}^N$ in $\RR^M$ is said to have the complement property if for every $I\subseteq \{1,\ldots,N\}$ either $\{\varphi_n\}_{n\in I}$ or $\{\varphi_n\}_{n\in I^c}$ span $\RR^M$. Notice, that if a family of vectors $\Phi$ is full spark, then it necessarily has the complement property. 

In Theorem 2.8 of \cite{BCE}, the authors show that vectors $\Phi=\{\varphi_n\}_{n=1}^N$ in $\RR^M$ allow phase retrieval if and only if $\Phi$ has the complement property.  While this proof leverages algebraic geometry, this result also follows from a short, elementary proof.  This was already noted in \cite{BCMN}, and we provide a proof as well since the result is quick and highlights the vectors which may cause phase retrieval to fail.  This general type of argument is used in later results and proofs.

\begin{theorem} \label{original}
A family of vectors $\Phi=\{\varphi_n\}_{n=1}^N$ in $\RR^M$ allow phase retrieval if and only if $\Phi$ has the complement property.  In particular, if $\Phi=\{\varphi_n\}_{n=1}^N$ is full spark then it gives phase retrieval.
\end{theorem}

\begin{proof}
Suppose $\Phi$ has the complement property, and let $x,y\in \RR^M$ satisfy $\abs{\inner{x}{\varphi_n}}=\abs{\inner{y}{\varphi_n}}$ for all $n=1, \ldots,N$.  Define $I=\{n:\inner{x}{\varphi_n}=\inner{y}{\varphi_n}\}$.  Then either $\{\varphi_n\}_{n\in I}$ or $\{\varphi_n\}_{n\in I^c}$ span $\RR^M$.   Suppose the first set spans.  Then $\{\varphi_n\}_{n\in I}$ contains a basis for $\RR^M$ and $\inner{x}{\varphi_n}=\inner{y}{\varphi_n}$ for $n\in I$ implies $x=y$.  If instead $\{\varphi_n\}_{n\in I^c}$ spans, we have $x=-y$, and we see $\cA$ is injective.

Next we prove the converse statement.  If $\Phi$ does not have the complement property, then there exists an $I\subset\{1,\ldots,N\}$ such that neither $\{\varphi_n\}_{n \in I}$ nor $\{\varphi_n\}_{n \in I^c}$ span $\RR^M$.  Choose $0\neq x\in [\spann\{\varphi_n\}_{n \in I}]^\perp$ and $0\neq y\in [\spann\{\varphi_n\}_{n \in I^c}]^\perp$.  So $x+y \not= \pm \left(x-y\right)$ while
\[
	\abs{\inner{x+y}{\varphi_n}}=\abs{\inner{x-y}{\varphi_n}}
\]
for all $n=1,\ldots,N$, and hence $\Phi$ does not allow phase retrieval.
\end{proof}

Since any family of vectors with the complement property in an $M$-dimensional space has at least $2M-1$ vectors, the minimal number of measurement vectors for phase retrieval in $\RR^M$ is $2M-1$.  At first glance one may be inclined to suggest the minimal number of subspaces required for $\cA$ to be injective is larger than $2M-1$.  Specifically, for a one dimensional space $W_n$, $\cA(x)(n)=\norm{P_nx}$ can come from only $\pm P_nx$.  For higher dimensional $W_n$, there is a continuum of $P_nx$ which give measurements $\cA(x)(n)=\norm{P_nx}$, and thus we appear to have less information in the subspace case.  This intuition is flawed however as we only care about $x$ as the pre-image of $\cA$ and not $P_nx$ as the pre-image under the norm.  We will in fact show $\cA$ can be injective with $2M-1$ subspaces, and we begin with a couple of lemmas.

\begin{lemma} \label{lem1}
Let $\{\varphi_n\}_{n=1}^N$ be full spark in an $M$-dimensional space.  Let $\{\psi_m\}_{m=1}^M$ be an orthonormal basis for the $M$-dimensional space constructed as follows:  Let $\psi_1$ be a random vector.  Then $\psi_2$ is chosen at random from $[span(\psi_1)]^{\perp}$.  Continue so that $\psi_k$ is chosen at random from $[span(\{\psi_n\}_{n=1}^{k-1})]^{\perp}$.  Then $\{\varphi_n\}_{n=1}^N \cup \{\psi_m\}_{m=1}^M$ is full spark with probability $1$.
\end{lemma}

\begin{proof}
For $1\leq k <M$, if $\{\varphi_n\}_{n=1}^N \cup \{\psi_m\}_{m=1}^{k}$ is full spark, and we desire $\{\varphi_n\}_{n=1}^N \cup \{\psi_m\}_{m=1}^{k+1}$ to be full spark, we must prove $\psi_{k+1}$ does not lie in the span of any $M-1$ vectors from $\{\varphi_n\}_{n=1}^N \cup \{\psi_m\}_{m=1}^{k}$.  Pick any such $M-1$ vectors, and denote this set by $A$.  Let $W_k:=[\spann(\{\psi_m\}_{m=1}^k)]^\perp$, and choose $\psi_{k+1}$ as a random unit norm vector from this $M-k$ dimensional space.  Then $\{\varphi_n\}_{n=1}^N \cup \{\psi_m\}_{m=1}^{k+1}$ is full spark if and only if $\psi_{k+1}\notin \spann(A)$, and this will be true with probability $1$ if and only if 
\begin{equation} \label{show}
	\dim(\spann(A)\cap W_k)\leq (M-k)-1.
\end{equation}
This follows because $\spann(A)\cap W_k$ is a subset of the $(M-k)$-dimensional space $W_k$ and if this inequality holds, then this intersection has measure zero. Hence with probability 1, $\psi_{k+1} \notin (\spann(A)\cap W_k)$ but $\psi_{k+1} \in W_k$.  We will prove (\ref{show}) by induction.

The first vector $\psi_1$ is chosen randomly from $W_0=\RR^M$.  If $A$ is any $M-1$ vectors in $\{\varphi_n\}_{n=1}^N$, we have
\[
	\dim(\spann(A)\cap W_0)=M-1
\]
so that $\{\varphi_n\}_{n=1}^N\cup \psi_1$ remains full spark with probability 1.  Now assume $\{\varphi_n\}_{n=1}^N \cup \{\psi_m\}_{m=1}^{k}$ is full spark.  Choose any $M-1$ vectors $A\subset \{\varphi_n\}_{n=1}^N\cup \{\psi_m\}_{m=1}^k$.  We consider two cases.

\noindent Case 1: Suppose $\psi_k\notin A$.  We may write
\[
	\spann(A)\cap W_k= \left(\spann(A) \cap W_{k-1}\right) \cap W_k.
\]
By our induction hypothesis, both subspaces on the right hand side have dimension less than or equal to $M-k$, and thus (\ref{show}) holds if we show the subspaces are not equal.  Note $\psi_k\notin A$ is needed to apply the induction hypothesis here.  Also note that $W_k \subseteq W_{k-1}$. For contradiction, suppose $\spann(A) \cap W_{k-1} = W_k$.  Switching to their orthogonal complements, since $\psi_k\in W_k^\perp$, we have $\psi_k\in  [\spann(A)\cap W_{k-1}]^\perp$ which is a space of dimension $k$.  Observing that $\psi_k\notin W_{k-1}^\perp$ which has dimension $k-1$ and
\[
	W_{k-1}^\perp \subset [\spann(A)\cap W_{k-1}]^\perp,
\]
it follows that $\psi_k$ lies in a unique one-dimensional space determined by $W_{k-1}$ and $A$.  Since $\psi_k$ was chosen randomly from an $M-k$ dimensional space, this fails with probability 1, and we have proven (\ref{show}).

\noindent Case 2: Suppose $\psi_k\in A$.  Since $\dim(W_k)=M-k$, note $\dim(\spann(A)\cap W_k)\leq M-k$.  For contradiction, suppose 
\begin{equation}\label{us1}
	\dim(\spann(A)\cap W_k)=M-k.
\end{equation}  
Then
\begin{equation}\label{sp}
	W_k\subset \spann(A).
\end{equation}
Choose some $\varphi\in \{\varphi_n\}_{n=1}^N$ where $\varphi\notin A$.  Then
\begin{equation}\label{us2}
	\dim(\spann(A\setminus \psi_k)\cap W_k) \leq \dim(\spann(A\setminus \psi_k \cup \varphi) \cap W_k )\leq (M-k)-1.
\end{equation}
where the last inequality follows by applying case 1.  Equations (\ref{us1}) and (\ref{us2}) imply
\[
	\dim(\spann(A\setminus \psi_k)\cap W_k)=(M-k)-1.
\]
However, since $\psi_k\perp W_k$ and $\psi_k\in A$, (\ref{us1}) and (\ref{sp}) imply 
\[
	\dim(\spann(A\setminus \psi_k)\cap W_k)=\dim(\spann(A)\cap W_k)=M-k,
\]
a contradiction.  We conclude (\ref{show}) must hold.
\end{proof}

By successive applications of Lemma \ref{lem1}, we immediately obtain the following result.

\begin{corollary}\label{orthspark}
Any finite number of randomly constructed orthonormal bases as in Lemma \ref{lem1} are full spark with probability 1.
\end{corollary}

\begin{lemma}\label{zerosandones}
Let $M \geq 2$ be a natural number. Choose any natural numbers $M-1\geq I_1\geq \cdots \geq I_M \geq 1$.  There exists a real invertible $M\times M$ matrix with entries zero and one which has precisly $I_k$ ones in the $k$-th row.
\end{lemma}

\begin{proof}
We proceed by induction on the dimension.  For $M=2$ the result is clear.  Assume the result holds for $M$, and consider $M+1$ so that for some natural number $s \leq M+1$ we have
\[
	M = I_1 =\cdots = I_s > I_{s+1} \geq \cdots \geq I_{M+1} \geq 1.
\]
Applying the induction hypothesis to $I_1-1 = \cdots = I_s-1 \geq I_{s+1} \geq \cdots \geq I_M$, we let $A=[a_{ij}]_{i,j=1}^M$ be an $M\times M$ invertible matrix with $I_k-1=M-1$ ones in row $k$ for $k=1,\ldots , s$ and $I_k$ ones in row $k$ for $k=s+1,\ldots , M$.  We now create an $(M+1)\times (M+1)$ matrix $B=[b_{ij}]_{i,j=1}^{M+1}$ by setting
\[
	b_{ij}=\begin{cases}
										a_{ij}	& 1\leq i,j \leq M \\
										1				&	1\leq i \leq s,\ j=M+1 \\
										1				& i=M+1,\ 1\leq j \leq I_{M+1} \\
										0				& \mbox{else}.
										\end{cases}
\]
Note $B$ has $I_k$ ones in each row for $k=1,\ldots,M+1$.  Since $A=[a_{ij}]_{i,j=1}^M=[b_{ij}]_{i,j=1}^M$ is invertible, we may row reduce $B$ to $\tilde{B}=[\tilde{b}_{ij}]_{i,j=1}^{M+1}$ where $[\tilde{b}_{ij}]_{i,j=1}^M=I_{M\times M}$ and row $M+1$ is left unchanged.  Suppose $\tilde{B}$ were not invertible.  Then row $M+1$ can be row reduced to all zeros, and examining the last entry in this row we must have
\begin{equation}\label{contra}
	\sum_{i=1}^{I_{M+1}} \tilde{b}_{i,M+1}=0.
\end{equation}
Now consider $\tilde{B}_{\ell}$ to be the matrix identical to $\tilde{B}$ but switch $\tilde{b}_{M+1,M+1}=0$ with $\tilde{b}_{M+1,\ell}=1$ where $\ell \in \{(1,\ldots,I_{M+1}\}$.  If $\tilde{B}_{\ell}$ is also non-invertible, we again may row reduce the last row to all zeros, and similar to (\ref{contra}), we now have
\begin{equation}\label{contra2}
	\sum_{\substack{i=1 \\ i\neq \ell}}^{I_{M+1}} \tilde{b}_{i,M+1}=-1.
\end{equation}
Notice (\ref{contra}) and (\ref{contra2}) imply $\tilde{b}_{M+1,\ell}=1$.  However, since this holds for all $\ell \in \{1,\ldots,I_{M+1}\}$ this contradicts (\ref{contra}).  It follows that either $\tilde{B}$ or $\tilde{B}_{\ell}$, for some $\ell \in \{1,\ldots,I_{M+1}\}$, must be invertible proving the result.
\end{proof}

We may now combine the previous lemmas to create a special case of the phase retrieval problem which reduces to the case of classical phase retrieval.

\begin{theorem}\label{phaseless}
Phase retrieval in $\RR^M$ is possible using $2M-1$ subspaces each of any dimension less than $M-1$.
\end{theorem}

\begin{proof}
Let $\{\varphi_n\}_{n=1}^{2M-1}$ be a family of vectors in $\RR^M$ with the complement property and the additional requirement that $\{\varphi_n\}_{n=1}^M$ and $\{\varphi_n\}_{n=M+1}^{2M-1}$ are orthonormal sets.  Such a set exists by Corollary \ref{orthspark}.  Let $I_k\subseteq \{1,\ldots,M\}$ for $k=1,\ldots, M$, let $J_k\subseteq \{M+1,\ldots,2M-1\}$ for $k=M+1,\ldots, 2M-1$, and let $P_{I_k}$ and $P_{J_k}$ denote the orthogonal projection onto $\spann(\{\varphi_n\}_{n\in I_k})$ and $\spann(\{\varphi_n\}_{n\in J_k})$ respectively.  We consider the problem of phase retrieval from $\norm{P_{I_k}x}$ and $\norm{P_{J_k}x}$ for $x \in \RR^M$ and for $k=1,\ldots,2M-1$.  

Let $A=[a_{kz}]_{k,z=1}^M$ be the $M\times M$ matrix whose rows correlate to $I_k$ where $a_{kz}=1$ if $z\in I_k$ and zero otherwise.  Define $B=[b_{kz}]_{k,z=1}^{M-1}$ similarly as the $(M-1) \times (M-1)$ matrix where $b_{kz}=1$ if $(z+M)\in J_k$ and zero otherwise.  We first examine the subspaces $\spann(\{\varphi_n\}_{n\in I_k})$ for $k=1,\ldots, M$. Notice for any $x \in \RR^M$,
\[
	\norm{P_{I_k}x}^2=\sum_{n\in I_k}\abs{\inner{x}{\varphi_n}}^2
\]
so that we have the equation
\begin{equation}\label{eq1}
		\begin{bmatrix}
			\norm{P_{I_1}x}^2 \\ \vdots \\ \norm{P_{I_M}x}^2
		\end{bmatrix}
		=
		A \begin{bmatrix}
				\abs{\inner{x}{\varphi_1}}^2 \\ \vdots \\ \abs{\inner{x}{\varphi_M}}^2
			\end{bmatrix}.
\end{equation}
Provided $A$ is invertible, we can solve for $\{\abs{\inner{x}{\varphi_n}}\}_{n=1}^M$.  We obtain a similar equation using $B$. So provided that $A$ and $B$ are both invertible, we can completely determine $\{\abs{\inner{x}{\varphi_n}}\}_{n=1}^{2M-1}$.  Now we have reduced the problem to the one-dimensional case, thus since $\{\varphi_n\}_{n=1}^{2M-1}$ were assumed to have the complement property then by Theorem \ref{original} it follows that phase retrieval is possible using the subspaces $\spann(\{\varphi_n\}_{n\in I_k})$ and $\spann(\{\varphi_n\}_{n\in J_k})$ for $k=1,\ldots, 2M-1$.  

All that remains is to pick $\{I_k\}_{k=1}^M$ and $\{J_k\}_{k=M+1}^{2M-1}$ so that $A$ and $B$ are invertible.  We may choose any invertible matrix with $I_k$ ones ($J_k$ ones respectively) in each row.  Note the number of ones in each row corresponds to the dimension of a subspace.  Such invertible matrices exist by Lemma \ref{zerosandones} for any $1\leq I_k \leq M-1$ and $1\leq J_k \leq M-2$.
\end{proof}

Questions concerning phase retrieval are fundamentally questions about the behavior of certain operators
in the $M(M+1)/2$-dimensional space of symmetric operators on $\HH_M$.  Although the above argument is
quite clean, it disguises what is really going on in the space of symmetric operators.  So we offer an alternative
proof of Theorem \ref{phaseless}  which makes the connections to the space of operators transparent. 
\vskip12pt

\noindent {\bf Proof 2 of Theorem \ref{phaseless}:}
\begin{proof}
Let $\{\varphi_n\}_{n=1}^{2M-1}$ be a family of vectors in $\RR^M$ which are full spark and have the additional requirement that $\{\varphi_n\}_{n=1}^M$ and $\{\varphi_n\}_{n=M+1}^{2M-1}$ are orthonormal sets.  Such a set exists by Corollary \ref{orthspark}.  Let $\HH^{M\times M}$ be the $M(M+1)/2$ dimensional vector space of $M\times M$ self-adjoint real matrices.
Consider the map $\RR^M$ into $\HH^{M \times M}$, given by $x \mapsto xx^*$.  Notice
\begin{equation}\label{prf2eqn}
	 \inner{xx^*}{\varphi_n \varphi_n^*}= Tr(xx^*\varphi_n\varphi_n^*)=\varphi_n^*xx^*\varphi_n=\abs{\inner{x}{\varphi_n}}^2.
\end{equation}
Also note that $\{\varphi_n\varphi_n^*\}_{n=1}^{2M-1}$ is linearly independent since $\Phi$ is full spark.  Further $\{\varphi_n\varphi_n^*\}_{n=1}^M$ and $\{\varphi_n\varphi_n^*\}_{n=M+1}^{2M-1}$ are orthonormal sets.  Thus for any $I\subset \{1,\ldots, M\}$ or $I\subset \{M+1,\ldots,2M-1\}$, letting $P_I$ be the orthogonal projection onto $span(\{\varphi_n\}_{n\in I})$, we see for any $x\in \RR^M$,
\[
	\norm{P_Ix}^2 = \norm{\sum_{n\in I}(\varphi_n\varphi_n^*)x}^2=\sum_{n\in I}\abs{\inner{x}{\varphi_n}}^2=\inner{xx^*}{\sum_{n\in I}\varphi_n\varphi_n^*}.
\]

Considering for a moment one dimensional projections, if there exist $x \neq \pm y$ such that $\abs{\inner{x}{\varphi_n}}=\abs{\inner{y}{\varphi_n}}$ for all $n\in \{1,\ldots,2M-1\}$, then equivalently by (\ref{prf2eqn})
\[
	\inner{xx^*-yy^*}{\varphi_n\varphi_n^*}=0
\]
for all $n$, or rather 
\[
	xx^*-yy^*\in (span\left\{\varphi_n\varphi_n^*\}_{n=1}^{2M-1}\right)^\perp.
\]
Thus if we can find subsets $I_k \subseteq \{1,\ldots,M\}$ for $k=1,\ldots, M$ and $J_k \subseteq \{M+1,\ldots,2M-1\}$ for $k=M+1,\ldots,2M-1$ such that
\[
	span\left(\left\{\sum_{n\in I_k}\varphi_n\varphi_n^*\right\}_{k=1}^M \bigcup \left\{\sum_{n\in J_k}\varphi_n\varphi_n^*\right\}_{k=M+1}^{2M-1} \right)=span\left(\{\varphi_n\varphi_n^*\}_{n=1}^{2M-1}\right),
\]
then Theorem \ref{original}
 guarantees no such $x$ and $y$ exist, and reconstruction without phase is possible using the subspaces $span(\{\varphi_n\}_{n\in I_k})$ and $span(\{\varphi_n\}_{n\in J_k})$ for $k=1,\ldots, 2M-1$.

We are left to find the sets $I_k$ and $J_k$.  We will choose to first examine the $I_k$.  Note we need only find $I_k$ such that $\{\sum_{n\in I_k}\varphi_n\varphi_n^*\}_{k=1}^{M}$ is linearly independent; the spanning properties then follow immediately since we would have linearly independent sets of the same size.  Suppose there exist scalars $c_k$ such that
\[
	\sum_{k=1}^{M} c_k \left(\sum_{n\in I_k}\varphi_n\varphi_n^*\right)=0.
\]
In terms of the basis $\{\varphi_n\varphi_n^*\}_{n=1}^M$ this is the equation
\[
	A_I \begin{bmatrix} c_1 &c_2 &\cdots &c_M \end{bmatrix}^t =0,
\]
where $A_I=[a_{kz}]_{k,z=1}^M$ is an $M\times M$ matrix with rows corresponding to $I_k$ such that $a_{kz}=1$ for all $z\in I_k$ and $0$ otherwise.  Thus $\{\sum_{n\in I_k}\varphi_n\varphi_n^*\}_{k=1}^{M}$ is linearly independent if $A$ is invertible. All that remains is to pick $\{I_k\}_{k=1}^M$ so that $A$ is invertible. We may choose any invertible $M \times M$ matrix with $I_k$ ones in each row.  Note the number of ones in each row corresponds to the dimension of a subspace.  Such invertible matrices exist by Lemma \ref{zerosandones} for any $1\leq I_k \leq M-1$.

A similar argument follows for $J_k$, except $A_J$ will be an $M-1\times M-1$ invertible matrix with $1 \leq J_k \leq M-2$ for all $k$.
\end{proof}

Notice we restrict the dimensions of the subspaces in this theorem to be less than $M-1$ since the matrix $B$ is $(M-1)\times (M-1)$ and we need $B$ to be invertible.  However, we can obtain phase retrieval in $\RR^M$ using subspaces of dimension less than $M$. To see this, suppose we are given a subspace $W_n$ such that dim$(W_n)=M-1$ and $W_n^\perp=\spann\{\varphi_n\}$. By considering the projection onto $W_n^\perp$, the intensity measurement $\abs{\inner{x}{\varphi_n}}$ should contain similar information to the measurement $\norm{P_nx}$ where here $P_n$ is the projection onto $W_n$.  Indeed since
\begin{equation} \label{complements}
	\norm{P_nx}^2=\norm{x}^2-\abs{\inner{x}{\varphi_n}}^2,
\end{equation} 
this suggests that the limits on dimension to less than $M-1$ in Theorem \ref{phaseless} could be relaxed.  Using notation from the proof, the matrix $A$ lets us solve for $\abs{\inner{x}{\varphi_n}}$ for $n=1,\ldots,M$ giving us $\norm{x}^2=\sum_{n=1}^M\abs{\inner{x}{\varphi_n}}^2$.  Now for the remaining subspaces corresponding to matrix $B$, we may indeed allow $M-1$ dimensional subspaces by considering instead orthogonal complements and using (\ref{complements}).  This leads to the following corollary.

\begin{corollary}\label{cor1}
Phase retrieval in $\RR^M$ is possible using $2M-1$ subspaces each of any dimension less than $M$.
\end{corollary}

We mention similar arguments hold for the complex case.  The authors of \cite{BBCE} show that $4M-2$ generic vectors allow phase retrieval in $\CC^M$ (we should point out $4M-2$ vectors are not necessary; \cite{BH} recently has given an example of $4M-4$ vectors which allow phase retrieval).   As Corollary \ref{orthspark} holds for complex vector spaces, we may obtain $4M-2$ full spark vectors, say $\{\varphi_n\}_{n=1}^{4M-2}$, which are the union of four orthogonal sets.  We then may create four matrices of zeros and ones as in the real case and reduce the problem of phase retrieval to the classical case with measurement vectors $\{\varphi_n\}_{n=1}^{4M-2}$.  Unfortunately phase retrieval with vector measurements in $\CC^M$ is fundamentally different from $\RR^M$, and there is no known necessary and sufficient condition for phase retrieval similar to the complement property.  The orthogonality requirements here destroys the genericity of our $4M-2$ vectors and with it the guarentee that $\{\varphi_n\}_{n=1}^{4M-2}$ allows phase retrieval.  

A recent result in \cite{MV} however, shows phase retrieval in $\CC^M$ is possible with the rows of four generic $M \times M$ unitary matrices.  Notice for any $x\in \CC^M$, by measuring with the $M$ rows of the first unitary matrix, we may determine $\norm{x}$.  At this point, measuring with any $M-1$ rows of another unitary determines that final measurement.  Therefore, this result actually implies phase retrieval is possible in $\CC^M$ with $4M-3$ vectors taken from $4$ orthonormal sets.  Taking these $4M-3$ vectors, the arguments above are now valid, and we have the following corollary for the complex case.

\begin{corollary}\label{com}
Phase retrieval in $\CC^M$ is possible using $4M-3$ subspaces each of any dimension less than $M$.
\end{corollary}

What we have done in this section is bound the minimal number of subspaces required for phase retrieval in $\RR^M$ by $2M-1$ and in $\CC^M$ by $4M-3$.  In the case of classical phase retrieval in $\RR^M$, $2M-1$ are also necessary.  It's unclear whether or not this holds true here as well.  In the next section we will characterize the subspaces which allow phase retrieval, and this characterization will highlight some of the difficulties in determining this answer.

\section{Characterizing subspaces which allow phase retrieval} \label{sec2}

Much recent advancement for classical phase retrieval has come from lifting the problem into the space of self-adjoint operators.  We may take a similar approach when using norms of projections as our measurements.  Let $\HH^{M\times M}$ be the $M(M+1)/2$ dimensional vector space of $M\times M$ self-adjoint real matrices.  Given a family of subspaces $\{W_n\}_{n=1}^N$ of $\RR^M$ with corresponding projections $P_n\in \HH^{M\times M}$, define the operator $F:\HH^{M\times M}\rightarrow \RR^N$ as $FA(n)=\inner{A}{P_n}_{HS}$.  Here $\inner{\ }{\ }_{HS}$ is the Hilbert-Schmidt inner product.  If we let $\{\varphi_{n,d}\}_{d=1}^{D_n}$ be an orthonormal basis for $W_n$, notice for any $x\in \RR^M$,
\[	
	F(xx^*)(n)	=\inner{xx^*}{P_n}_{HS}=\Tr(xx^*P_n)=\Tr(xx^*\sum_{d=1}^{D_n}\varphi_{n,d}\varphi_{n,d}^*),
\]
and by the cyclic property of the trace
\[
		\Tr(xx^*\sum_{d=1}^{D_n}\varphi_{n,d}\varphi_{n,d}^*)=\sum_{d=1}^{D_n}\varphi_{n,d}^*xx^*\varphi_{n,d}=\sum_{d=1}^{D_n}\abs{\inner{x}{\varphi_{n,d}}}^2=\norm{P_nx}^2.
\]

Therefore $F(xx^*)(n)=\norm{P_nx}^2$, and much like the classical phase retrieval problem \cite{BCE, BCMN, CSV}, we may linearize the measurements by working in this higher dimensional space of self-adjoint operators.  This identification yields a useful characterization for when subspaces allow phase retrieval.  For classical phase retrieval, Lemma 9 in \cite{BCMN} provides this characterization.

\begin{lemma} (Lemma 9 in \cite{BCMN}) \label{lemma1}
Let $\Phi=\{\varphi_n\}_{n=1}^N$ be a family of vectors in $\RR^M$.  Then $\Phi$ allows phase retrieval if and only if the null space of $G:\HH^{M\times M}\rightarrow \RR^N$ given by $GA(n)=\inner{A}{\varphi_n\varphi_n^*}_{HS}$ does not contain a matrix of rank $1$ or $2$.
\end{lemma}

If we generalize this result to projections of arbitrary ranks, it turns out that the characterization is identical.  In fact, the same proof technique holds by replacing the operator $G$ in Lemma \ref{lemma1} with its subspace component analog $F$.  For this reason, we give the following as a corollary and omit the proof.

\begin{corollary}\label{char}
Given subspaces $\{W_n\}_{n=1}^N$ in $\RR^M$ with corresponding projections $P_n$, $\{W_n\}_{n=1}^N$ allows phase retrieval if and only if there are no matrices of rank $1$ or $2$ in the null space of $F$.
\end{corollary}

Since we know $2M-1$ vectors are necessary for phase retrieval, one would hope that the close similarities between the characterizations in Lemma \ref{lemma1} and Corollary \ref{char} would provide insight into the minimal number of subspaces required for phase retrieval.  Unfortunately it is difficult to draw any comparison between the two problems in this regard.  The main issue here is that the space of rank 1 and rank 2 operators do not form a subspace in $\HH^{M\times M}$, and null spaces of $F$ and $G$ may (or may not) intersect this space in fundamentally different ways.  The minimal number $2M-1$ arises for phase retrieval with vector measurements since this is the fewest number of vectors which may have the complement property in $\RR^M$. We will now develop a characterization of subspaces $\{W_n\}_{n=1}^N$ allowing phase retrieval which is akin to the complement property; but this also falls short to providing a minimal number of subspaces required.  To accomplish this, we give a few preliminary results, the first of which follows from Lemma \ref{lemma1} and Corollary \ref{char}.

\begin{corollary}\label{cor2}
Suppose $\{W_n\}_{n=1}^N$ are subspaces allowing phase retrieval for $\RR^M$.  If $\{\varphi_{n,d}\}_{d=1}^{D_n}$ is an orthonormal basis for $W_n$ for each $n=1,\ldots,N$, then $\Phi=\{\varphi_{n,d}\}_{n=1,d=1}^{N,D_n}$ allows phase retrieval for $\RR^M$.
\end{corollary}

\begin{proof}
For $\Phi=\{\varphi_{n,d}\}_{n=1,d=1}^{N,D_n}$, consider the operator $G$ as in Lemma \ref{lemma1}.  If $A\in \Nul(G)$ then $\inner{A}{\varphi_{n,d}\varphi_{n,d}^*}_{HS}=0$ for all choices of $n$ and $d$.  Since $P_n=\sum_{d=1}^{D_n}\varphi_{n,d}\varphi_{n,d}^*$, then $\inner{A}{P_n}_{HS}=0$ implying $A\in \Nul(F)$.  Thus $\Nul(G)\subseteq \Nul(F)$ and the result follows from Lemma \ref{lemma1} and Corollary \ref{char}.
\end{proof}

\begin{lemma}\label{lem3}
Let $P$ be the orthogonal projection onto an $N$-dimensional subspace $W\subseteq \RR^M$.  Given $x,y\in \RR^M$ the following are equivalent:
\begin{enumerate}[(a)]
	\item $\norm{Px}=\norm{Py}$ \label{a}
	\item There exists an orthonormal basis $\{\varphi_n\}_{n=1}^N$ for $W$ such that $\abs{\inner{x}{\varphi_n}}=\abs{\inner{y}{\varphi_n}}$ for all $n=1,\ldots,N$. \label{b}
\end{enumerate}
\end{lemma}

\begin{proof}
(\ref{a})$ \Rightarrow $(\ref{b}):  Consider the vectors $Px,Py\in W$.  We may assume $Px\neq \pm Py$ as the other case is trivial.  Let
\[
	\varphi_1:= \frac{Px+Py}{\norm{Px+Py}}, \  \ \varphi_2:= \frac{Px-Py}{\norm{Px-Py}}.
\]
Letting $c=1/(\norm{Px+Py} \norm{Px-Py})$,
\[
	\inner{\varphi_1}{\varphi_2}=c \inner{Px+Py}{Px-Py}= c (\norm{Px}^2-\norm{Py}^2 + \inner{Py}{Px}-\inner{Px}{Py})=0.
\]
Given these two orthonormal vectors, take $\{\varphi_n\}_{n=1}^N$ to be any completion of $\{\varphi_1,\varphi_2\}$ to an orthonormal basis for $W$. Note that $Px, Py \in \mbox{span}\{\varphi_1,\varphi_2\}$. Thus since $\{\varphi_i\}_{i=1}^N$ is an orthonormal basis for $W$ then 
\[
	\inner{x}{\varphi_n}=\inner{x}{P\varphi_n}=\inner{Px}{\varphi_n}=0, 
\] 
and similarly $\inner{y}{\varphi_n}=0$ for all $n=3,\ldots,N$.  We also have
\begin{align*}
	\abs{\inner{x}{Px + Py}}	&=\abs{\norm{Px}^2+ \inner{Px}{Py}}=\abs{\norm{Py}^2+ \inner{Py}{Px}}\\
															&=\abs{\inner{Py}{Py + Px}}=\abs{\inner{y}{Px + Py}}
\end{align*}
and similarly $\abs{\inner{x}{Px - Py}}=\abs{\inner{y}{Px - Py}}$ for $n=1,2$.  Hence $\abs{\inner{x}{\varphi_n}}=\abs{\inner{y}{\varphi_n}}$ for all $n=1,\ldots,N$.

(\ref{b})$ \Rightarrow $(\ref{a}): This is immediate since
\[
	\norm{Px}^2=\sum_{n=1}^N \abs{\inner{\varphi_n}{Px}}^2=\sum_{n=1}^N \abs{\inner{\varphi_n}{x}}^2=\sum_{n=1}^N \abs{\inner{\varphi_n}{y}}^2=\norm{Py}^2.
\]
\end{proof}

Combining Corollary \ref{cor2} and Lemma \ref{lem3}, we arrive at a characterization for when $\{W_n\}_{n=1}^N$ allows phase retrieval in $\RR^M$ in terms of the complement property.

\begin{theorem} \label{main}
Let $\{W_n\}_{n=1}^N$ be subspaces of $\RR^M$.  The following are equivalent:
\begin{enumerate}[(a)]
	\item $\{W_n\}_{n=1}^N$ allows phase retrieval for $\RR^M$. \label{aa}
	\item For every orthonormal basis $\{\varphi_{n,d}\}_{d=1}^{D_n}$ of $W_n$, the set $\{\varphi_{n,d}\}_{n=1, d=1}^{N, \ D_n}$ allows phase retrieval in $\RR^M$ and thus has the complement property. \label{bb}
\end{enumerate}
\end{theorem}

\begin{proof}
(\ref{aa})$\Rightarrow$ (\ref{bb}): This is Corollary \ref{cor2}.

(\ref{bb})$\Rightarrow$ (\ref{aa}):  Suppose we have $x,y \in \RR^M$ such that $\norm{P_nx}=\norm{P_ny}$ for all $n=1,\ldots,N$.  By Lemma \ref{lem3}, we may choose an orthonormal basis $\{\varphi_{n,d}\}_{d=1}^{D_n}$ for each $W_n$ so that $\abs{\inner{x}{\varphi_{n,d}}}=\abs{\inner{y}{\varphi_{n,d}}}$ for all $d=1,\ldots, D_n$ and all $n=1,\ldots,N$.  Since our assumption guarantees $\{\varphi_{n,d}\}_{n=1, d=1}^{N, \ D_n}$ allows phase retrieval, we have $x=\pm y$.  Thus $\{W_n\}_{n=1}^N$ must allow phase retrieval.
\end{proof}

The complement property is a convenient property with which to work, and this is why we present the above theorem in terms of orthonormal bases.  However, the proof of this theorem doesn't require us to consider only vectors.  Instead, the same general arguments hold if we take each $W_n$ and split this subspace into orthogonal subspaces which span $W_n$. 

\begin{corollary}
Let $\{W_n\}_{n=1}^N$ be subspaces of $\RR^M$.  The following are equivalent:
\begin{enumerate}[(a)]
	\item $\{W_n\}_{n=1}^N$ allow phase retrieval for $\RR^M$.
	\item For every choice of orthogonal subspaces $\{Z_{n,d}\}_{d=1}^{D_n}$ where  $\bigoplus_{d=1}^{D_n}Z_{n,d}=W_n$, the subspaces $\{Z_{n,d}\}_{n=1,d=1}^{N,\ \ D_n}$ allow phase retrieval in $\RR^M$. 
\end{enumerate}
\end{corollary}

\section{Additional properties of subspaces regarding phase retrieval} \label{sec3}

At this point we have given several abstract characterizations of the subspaces $\{W_n\}_{n=1}^N$ which allow phase retrieval in $\RR^M$.  Note however the only subspaces $\{W_n\}_{n=1}^N$ which we have concretely shown to satisfy these characterizations are highly structured.  That is, the only subspaces which we have shown to allow phase retrieval are those constructed in Theorem \ref{phaseless}.  For the special case when $\{W_n\}_{n=1}^N$ are hyperplanes, we will overcome this restriction of structure and produce highly non-structured subspaces which alllow phase retrieval in section \ref{sec4}. In general however, we believe any $2M-1$ random subspaces should also allow phase retrieval much as $2M-1$ random vectors allow phase retrieval \cite{BBCE}.  While we cannot prove this, we take an incremental step by showing the subspaces which allow phase retrival are open in some sense.  Specifically, we show when given subspaces $\{W_n\}_{n=1}^N$ which allow phase retrieval, there exist open balls $B_n(W_n,\epsilon)$ around each $W_n$ such that for any $W'_n\in B_n(W_n,\epsilon)$, the subspaces $\{W'_n\}_{n=1}^N$ allow phase retrieval.  We again require a few preliminary results to build to this end.  

First we show that subspaces $\{W_n\}_{n=1}^N$ allow phase retrieval if and only if we cannot  find nonzero orthogonal vectors $x,y\in \RR^M$ such that $\cA x=\cA y$.  This is useful in that to show $\{W_n\}_{n=1}^N$ allows phase retrieval, we need only show $\cA(x)\neq \cA(y)$ for all $x\perp y$ rather than all $x\neq \pm y$.  We note that the argument used to prove this result may be extracted from the arguments needed for Lemma \ref{lemma1}.

\begin{lemma}\label{orth}
Subspaces $\{W_n\}_{n=1}^N$ do not allow phase retrieval if and only if there exists nonzero $u,v \in \RR^M$ with $u\perp v$ such that $\norm{P_nu}=\norm{P_nv}$ for all $n=1,\ldots, N$.
\end{lemma}

\begin{proof}
The necessity direction is obvious, so for sufficiency, suppose $\{W_n\}_{n=1}^N$ do not allow phase retrieval.  Then there exists nonzero $x,y \in \RR^M$ with $x\neq \pm y$ such that $\norm{P_nx}=\norm{P_ny}$ for all $n=1,\ldots,N$.  In the operator space, this implies $F(xx^*)=F(yy^*)$, so that $xx^*-yy^*$ is in the null space of $F$, where $F$ is the linear operator as defined in the beginning of Section 3.  Note that $xx^*-yy^*$ is a rank two, symmetric operator. Thus by the Spectral Theorem, there exist orthogonal eigenvectors $u,v \in \RR^M$ and nonzero scalars $\lambda_1,\lambda_2 \in \RR$ such that $xx^*-yy^*=\lambda_1uu^*+\lambda_2vv^*$.  Then for all $n=1,\ldots,N$ we have,
\[
	0=F(\lambda_1uu^*+\lambda_2vv^*)(n)=\lambda_1 \norm{P_n u}^2+\lambda_2 \norm{P_n v}^2.
\] 

Since $u,v$ are nonzero, it follows that $\norm{P_n u}^2>0$ and $\norm{P_n v}^2>0$. Hence $\lambda_1$ and $\lambda_2$ must have opposite signs, which implies that 
\[
	\abs{\lambda_1}\norm{P_n u}^2=\abs{\lambda_2} \norm{P_n v}^2.
\]
Thus $u/\sqrt{\abs{\lambda_1}}$, and  $v/\sqrt{\abs{\lambda_2}}$ are the orthogonal vectors we seek.
\end{proof}

\begin{lemma}\label{lemon}
Suppose $\{W_n\}_{n=1}^N$ are subspaces allowing phase retrieval for $\RR^M$.  Then there exists a $\delta>0$ such that for any $x\perp y$ where $1=\norm{x}\geq \norm{y}>0$, there exists an $1\leq n \leq N$ such that
\[
	\abs{\norm{P_nx}-\norm{P_ny}}>\delta.
\]
\end{lemma}

\begin{proof}
We proceed by way of contradiction.  Assume $\{W_n\}_{n=1}^N$ allow phase retrieval, and for every $j=1,2,\ldots$ there exists $x_j\perp y_j$ where $1=\norm{x_j}\geq \norm{y_j}>0$ so that
\begin{equation}\label{allow}
		\abs{\norm{P_nx}-\norm{P_ny}}\leq\ \frac{1}{j}
\end{equation}
for every $n=1,\ldots,N$.  By switching to a subsequence, we may assume $x_j\rightarrow x$ with $\norm{x}=1$.  Note there exists some $n$ such that $P_nx\neq 0$ for otherwise $\{W_n\}_{n=1}^N$ would not allow phase retrieval.  By (\ref{allow}), this also implies $P_n y_j$ does not converge to zero for some $n$ and thus $y_j$ cannot converge to zero.  We therefore switch to a further subsequence such that $x_n\rightarrow x$, $y_n\rightarrow y$, $1=\norm{x}\geq \norm{y}>0$, and $x\perp y$.  Moreover, equation (\ref{allow}) now implies $\norm{P_nx}=\norm{P_ny}$ for all $n=1,\ldots,N$.  We conclude $\{W_n\}_{n=1}^N$ does not allow phase retrieval - a contradiction.
\end{proof}

Combining these lemmas, we have the desired theorem.

\begin{theorem} \label{open}
Suppose $\{W_n\}_{n=1}^N$ are subspaces allowing phase retrieval for $\RR^M$.  Let $\{W'_n\}_{n=1}^N$ be subspaces with associated orthogonal projections $Q_n$.  Then there exists an $\epsilon>0$ such that when $\norm{P_n-Q_n}<\epsilon$ for all $n=1,\ldots,N$, then $\{W'_n\}_{n=1}^N$ allow phase retrieval.
\end{theorem}

\begin{proof}
By Lemma \ref{orth}, it suffices to prove that for any nonzero $x\perp y$ there exists some $n$ such that $\norm{Q_nx} \neq \norm{Q_ny}$.  Take any nonzero $x\perp y$, and we may assume by scaling and switching $x$ and $y$ if necessary that $1=\norm{x}\geq \norm{y}>0$.  By Lemma \ref{lemon} there exists a $\delta>0$ such that
\[
	\abs{\norm{P_{n_0}x}-\norm{P_{n_0}y}}>\delta
\]
for some $n_0 \in \{1,.\ldots,N\}$.  Then
\begin{align*}
	\abs{\norm{Q_{n_0}x}-\norm{Q_{n_0}y}}	&\geq \abs{\norm{P_{n_0}x}-\norm{P_{n_0}y}}-\abs{\norm{Q_{n_0}x}-\norm{P_{n_0}x}}-\abs{\norm{P_{n_0}y}-\norm{Q_{n_0}y}}\\
																&> \delta - \norm{P_{n_0}x-Q_{n_0}x}-\norm{P_{n_0}y-Q_{n_0}y} \\
																&\geq \delta - 2\norm{P_{n_0}-Q_{n_0}}\\
																&> \delta - 2\epsilon \\
																&>0
\end{align*}
when $\epsilon<\delta/2$.
\end{proof}

\section{Open problems and future work} \label{sec4}
In this final section we highlight several open questions concerning phase retrieval from the norms of subspace components.  Some of these questions arose in earlier sections, and some we present here for the first time.  We give the main problem we wish to be answered and then incremental progress towards a solution.

\begin{problem}\label{P1}
What is the minimal number $N$ such that $\{W_n\}_{n=1}^N$ allow phase retrieval in $\RR^M$?  Does this number depend upon the dimensions of the subspaces?
  \end{problem}

Theorem \ref{phaseless} showed we can use $2M-1$ subspaces, but the characterization in Theorem \ref{main} may suggest we could possibly use fewer.  For example, in $\RR^4$, we can consider $2(4)-2=6$ subspaces $\{W_n\}_{n=1}^6$ each of dimension $2$.  Let's suppose we choose these subspaces so that $W_n\cap W_{n'}=0$ for all $n\neq n'$ so that any two subspaces span $\RR^4$.  If $\{W_n\}_{n=1}^N$ do not allow phase retrieval, there must exist orthonormal bases $\{\varphi_n,\psi_n\}$ for each $W_n$ such that $\cup_{n=1}^6\{\varphi_n,\psi_n\}$ does not posses the complement property.  We recall this means there exists a partition of $\cup_{n=1}^6\{\varphi_n,\psi_n\}$ into two non-spanning sets.  

The existence of such a set without the complement property can be reformulated into a statement about subspaces.  For any three dimensional subspace $Z\subset \RR^4$, notice $\dim (Z\cap W_n)\geq 1$ for all $n=1,\ldots,6$.  If we choose unit norm vectors $\{\varphi_n\}_{n=1}^6$ where $\varphi_n\in W_n\cap Z$, then these uniquely determine $\psi_n$ up to sign.  Letting $I\subseteq \{1,\ldots,6\}$ be such that $\psi_n \notin Z$ for $n\in I$, we may find orthonormal bases without the complement property if there exists a $Z$ such that
\[
	\dim \spann(\{\psi_n\}_{n\in I})\leq 3.
\]
Since any two $W_n$ span all of $\RR^4$, $Z$ can contain at most one complete pair $\{\varphi_n,\psi_n\}$ implying $\abs{I}\geq 5$.  It is a very strong property that these $5$ vectors all lie in a $3$-dimensional subspace, and it does not seem unreasonable that there could exist choices of $\{W_n\}_{n=1}^6$ such that this property does not hold for any $Z$.  That is, it does not appear unreasonable that phase retrieval could be accomplished with fewer than $2M-1$ subspaces.

Interestingly the characterization in Corollary \ref{char} suggests to the contrary.  That is, Corollary \ref{char} suggests phase retrieval in $\RR^M$ requires at least $2M-1$ subspaces.  Details may be found in \cite{BCMN} where the reasoning is similar to that for the ``$4M-4$ conjecture'' for complex phase retrieval.  Briefly, the intuition is that the rank $1$ and rank $2$ operators in $\HH^{M\times M}$ form a real projective variety of dimension $2M-1$.  Recall the operator $G$ from Lemma \ref{lemma1}. The projective dimension theorem and rank-nullity theorem would require the null space of $G$ to intersect non-trivially with this real projective variety when 
\[
	\dim(\mbox{Null}(G)) +(2M-1) > \dim(\HH^{M \times M})=\frac{M(M+1)}{2}.
\]

Hence to have a trivial intersection and thus injectivity of $G$ or equivalently allow phase retrieval then we need 
\[ 
\dim(\mbox{Null}(G)) +(2M-1) \leq \dim(\HH^{M \times M})
\]
Or equivalently, 
\[
(2M-1) \leq \dim(\HH^{M \times M}) - \dim(\mbox{Null}(G))= \mbox{rank}(G) \leq N.
\]

Thus to allow phase retrieval, $N\geq 2M-1$.  This intuition does not yield a proof since the projective dimension theorem does not apply over the non-algebraically closed field $\RR$.  This problem thus appears similar to the complex phase retrival problem where the minimal number of vectors required for phase retrieval has proven a very difficult problem.

\begin{problem}\label{P2}
 Can phase retrieval be done with random subspaces of $\RR^M$?
 \end{problem}

Theorem \ref{open} showed if $\{W_n\}_{n=1}^N$ allowed phase retrieval, then we may replace any $W_n$ with another subspace $W'_n$ from a small open ball around $W_n$ and this new collection of subspaces allow phase retrieval.  However, our concretely constructed sets $\{W_n\}_{n=1}^N$ which allow phase retrieval are highly structured according to Theorem \ref{phaseless}.  In the case of phase retrieval with vector measurements, the sets of vectors which allow phase retrieval form an open dense set.  It remains to show that we can still do phase retrieval when $W_n$ is replaced with any $W'_n$ from some open dense set in an analog to the one dimensional case.

First observe there exists subspaces $\{W_n\}_{n=1}^N$ allowing phase retrieval such that $\{W_n^\perp\}_{n=1}^N$ do not.
\begin{example}
Let $\{\varphi_n\}_{n=1}^3$ and $\{\psi_n\}_{n=1}^3$ be orthonormal bases for $\RR^M$ such that $\{\varphi_n\}_{n=1}^3 \cup \{\psi_n\}_{n=1}^3$ is full spark.  Consider the subspaces
\begin{align*}
	W_1&=\spann(\{\varphi_1,\varphi_3\}) &&W_1^\perp=\spann(\{\varphi_2\})\\
	W_2&=\spann(\{\varphi_2,\varphi_3\}) &&W_2^\perp=\spann(\{\varphi_1\})\\
	W_3&=\spann(\{\varphi_3\}) 					 &&W_3^\perp=\spann(\{\varphi_1,\varphi_2\})\\
	W_4&=\spann(\{\psi_1\}) 					 	 &&W_4^\perp=\spann(\{\psi_2,\psi_3\})\\
	W_5&=\spann(\{\psi_2\}) 						 &&W_5^\perp=\spann(\{\psi_1,\psi_3\}).
\end{align*}
Then $\{W_n\}_{n=1}^5$ allows phase retrieval for $\RR^3$ while the orthogonal complements $\{W_n^\perp\}_{n=1}^5$ do not.
\end{example}

To see this, notice the subspaces $\{W_n\}_{n=1}^5$ allow phase retrieval from a direct application of Theorem \ref{phaseless}.  Considering the orthogonal complements $\{W_n^\perp\}_{n=1}^5$ with associated orthogonal projections $\{Q_n\}_{n=1}^5$, notice $Q_1+Q_2=Q_3$.  Thus the measurement $\norm{Q_3x}^2$ does not contribute any new information (or think of $Q_3$ as a linearly dependent operator which when removed does not change the null space associated with $F$ as in Corollary \ref{char}).  Thus $\{W_n^\perp\}_{n=1}^5$ allows phase retrieval if and only if $\{W_n^\perp\}_{n\in \{1,2,4,5\}}$ allows phase retrieval.  However, for the special case of $\RR^3$, we can show $5$ subspaces are necessary.

Indeed, suppose we have any subspaces $\{W_n\}_{n=1}^N$ in $\RR^3$ and corresponding projections $\{P_n\}_{n=1}^N$ with $N\leq 4$.  Since $\dim (\HH^{3\times 3})= 6$, by the rank-nullity theorem, the dimension of the null space of $F$ is greater than or equal to $2$. Hence there exist two nonzero, linearly independent matrices $A,B \in \HH^{3 \times 3}$ such that $F(A)=F(B)=0$, where $F$ is the operator as defined in the begining of section 3.  If either matrix is rank $1$ or $2$, then $\{W_n\}_{n=1}^N$ in $\RR^3$ do not allow phase retrieval by Corollary \ref{char}.  So assume $A$ and $B$ are full rank and consider the continuous map
\[
	f:t\mapsto \det (A\cos t+B \sin t), \ \   t\in[0,\pi].
\]
Since $f(0)=\mbox{det}(A)\neq 0$ and $f(\pi)=\mbox{det}(-A)=(-1)^3\mbox{det}(A)=-\mbox{det}(A)\neq 0$, then by the intermediate value theorem there exists some $t_0 \in [0,\pi]$ such that $f(t_0)=\det(A\cos t_0+B \sin t_0)=0$.  Therefore $C:=A\cos t_0+B \sin t_0$ is a rank $1$ or $2$ matrix, such that $F(C)=0$. Note that $C \neq 0$ since $A$ and $B$ are nonzero, linearly independent matrices. Therefore $C$ is a nonzero, rank 1 or 2 matrix in the null space of $F$ and thus by Corollary \ref{char}, $\{W_n\}_{n=1}^N$ in $\RR^3$ again fails phase retrieval.

There are also special cases when phase retrieval is always possible with orthogonal complements.

\begin{theorem}\label{theorem10}
Suppose $\{W_n\}_{n=1}^N$ are subspaces of $\RR^M$ allowing phase retrieval with corresponding orthogonal projections $\{P_n\}_{n=1}^N$.  If $I\in \spann(\{P_n\}_{n=1}^N)$ so that $I=\sum_{n=1}^N a_nP_n$ with $\sum_{n=1}^N a_n \neq 1$, then $\{W_n^\perp\}_{n=1}^N$ allow phase retrieval.
\end{theorem}

\begin{proof}
Note
\[
	\sum_{n=1}^N a_n(I-P_n)=\left(\sum_{n=1}^Na_nI\right)-I=\left(\left(\sum_{n=1}^Na_n \right)-1\right)I,
\]
so letting $b=\left(\sum_{n=1}^Na_n\right)-1$, we have $I=\sum_{n=1}^N\frac{a_n}{b}(I-P_n)$.  Thus the measurements $\norm{(I-P_n)x}$ associated with $\{W_n^\perp\}_{n=1}^N$ allow one to determine
\[
	\norm{x}^2=\inner{xx^*}{I}_{HS}=\inner{xx^*}{\sum_{n=1}^N\frac{a_n}{b}(I-P_n)}_{HS}=\sum_{n=1}^N\frac{a_n}{b}\inner{xx^*}{I-P_n}_{HS}=\sum_{n=1}^N\frac{a_n}{b}\norm{(I-P_n)x}^2.
\]
Since $\{W_n\}_{n=1}^N$ allow phase retrieval and $\norm{P_nx}^2=\norm{x}^2-\norm{(I-P_n)x}^2$, it follows that $\{W_n^\perp\}_{n=1}^N$ allows phase retrieval.
\end{proof}

Notice that when dim$(W_n)=1$ for all $n=1,\ldots,N$, Theorem \ref{theorem10} classifies when hyperplanes allow phase retrieval. In light of these unstructured hyperplanes allowing phase retrieval, it is interesting to note that given subspaces $\{W_n\}_{n=1}^N$ which do not allow phase retrieval in $\RR^M$, it is always possible to find hyperplanes $\{W'_n\}_{n=1}^N$ such that $W_n\subseteq W'_n$ where $\{W'_n\}_{n=1}^N$ do not allow phase retrieval.

\begin{proposition}
If the subspaces $\{W_n\}_{n=1}^N$ do not allow phase retrieval in $\RR^M$ then there exists $\{W'_n\}_{n=1}^N$ not allowing phase retrieval where $\dim W'_n=M-1$ and $W_n\subseteq W'_n$ for all $n=1,\ldots,N$.
\end{proposition}

\begin{proof}
Since $\{W_n\}_{n=1}^N$ do not allow phase retrieval, there exists nonzero $x,y \in \RR^M$ such that $x\neq \pm y$ and $\norm{P_nx}=\norm{P_ny}$ for all $n=1,\ldots,N$.  For $n$ such that $\dim W_n= M-1$, let $W'_n=W_n$.   For any other $n$, say $\dim W_n=D_n\leq M-2$, we construct $W'_n$ as follows.  Let $\{\varphi_1,\varphi_2\}$ be orthonormal vectors in $W_n^\perp$.  Let $Z:=\spann(\{\varphi_{1},\varphi_{2}\})$, with $P_Z:\RR^M\rightarrow Z$ an orthogonal projection.  Set
\[
	u:= P_Z x \ \ \mbox{and} \ \ v:=P_Z y,
\]
and consider the function $f:Z\rightarrow \RR$ given by 
\[
	f(z)=\abs{\inner{u}{z}}-\abs{\inner{v}{z}}.
\]
Let $z_1,z_2\in Z$ be unit norm vectors such that $z_1\perp u$ and $z_2 \perp v$.  Then  $f(z_1)\leq 0\leq f(z_2)$, and by the intermediate value theorem, there exists a $z_0\in Z$ where $f(z_0)=0$ and hence $\abs{\inner{u}{z_0}}=\abs{\inner{v}{z_0}}$.  We assume $z_0\neq 0$.  Note if $z_0=0$, then $z_1=-z_2$, $f(z_1)=0=f(z_2)$, and we could instead choose $z_0=z_1\neq 0$.
Letting $W'_n=\spann(\{W_n,z_0\})$ with corresponding orthogonal projection $P_n'$, we have
\begin{align*}
	\norm{P'_nx}^2	&=\norm{P_nx}^2+\abs{\inner{x}{\frac{z_0}{\norm{z_0}}}}^2 
									= \norm{P_nx}^2+\abs{\inner{x}{P_Z\frac{z_0}{\norm{z_0}}}}^2
									=\norm{P_nx}^2+\abs{\inner{u}{\frac{z_0}{\norm{z_0}}}}^2\\
									&=\norm{P_ny}^2+\abs{\inner{v}{\frac{z_0}{\norm{z_0}}}}^2
									=\norm{P_ny}^2+\abs{\inner{y}{P_Z\frac{z_0}{\norm{z_0}}}}^2
									=\norm{P'_ny}^2.
\end{align*}
It follows that $\{W'_n\}_{n=1}^N$ do not allow phase retrieval.  We now may iterate this argument until $\dim W'_n = M-1$ for all $n=1,\ldots N$.

\end{proof}

\begin{problem}\label{P4}
Show random subspaces or find examples of non-structured subspaces of arbitrary dimension which allow phase retrieval.
\end{problem}

For the special case of hyperplanes, Theorem \ref{theorem10} allows us to construct highly non-structured subspaces which allow phase retrieval.

It is known that for any $N>M$, the full spark families of vectors $\{\varphi_n\}_{n=1}^N$ are a
dense, open set of full measure
 within the families of vectors $\{\varphi_n\}_{n=1}^N$ such that $\sum_{n=1}^N \varphi_n\varphi_n^*=I$ \cite{CK}.  
 From here it is easy to construct full spark vectors $\{\varphi_n\}_{n=1}^N$ in $\RR^M$ with $\sum_{n=1}^N \varphi_n\varphi_n^*=I$ such that no two vectors are orthogonal.  It follows that $\{\frac{\varphi_n}{\norm{\varphi_n}\}}\}_{n=1}^N$ is
 full spark and thus allows phase retrieval.  Letting $P_n$ be the orthogonal projection onto $\spann(\varphi_n)$, we have $P_n=\frac{\varphi_n \varphi^*_n}{\norm{\varphi_n}^2}$ so
\[ 
	\sum_{n=1}^N \norm{\varphi_n}^2 P_n = I,
\]
and by Theorem \ref{theorem10}, it follows that $\{W_n = (I-P_n)\RR^M\}_{i=1}^N$ is a family of hyperplanes allowing phase retrieval.  These are {\it unstructured} since structured subspaces would have the property that their orthogonal complements contain a large number of orthogonal vectors.

\begin{remark}
The example above is not an equal norm Parseval frame.  The problem here is that
we do not know if the full spark equal norm Parseval frames are dense in the
equal norm Parseval frames (See \cite{CK}).  However, in concrete cases we can achieve this.
For example, in $\RR^2$ choose 5 equally spaced vectors $\{x_i\}_{i=1}^5=\{(x_{i1},x_{i2})\}_{i=1}^5$
on a circle of radius $\sqrt{\frac{2}{5}}$ and let $z^2=\frac{1}{5}$. Then the
vectors in $\RR^3$ given by $\{(x_{i1},x_{i2},z)\}_{i=1}^5$ form a full spark Parseval frame
for $\RR^3$ and hence the above example works.
\end{remark}

\begin{remark}
The above example can be generalized to Parseval fusion frames.  However, 
much less is known about Parseval fusion frames with arbitrary dimensional
subspaces.  We do not want
to introduce this subject here and this will be covered in a later paper.
\end{remark}

The following problem is also open at this time.

\begin{problem}\label{P5}
Find examples of (classify) the subspaces $\{W_n\}_{n=1}^N$ which allow phase retrieval but the span of their associated projections $\{P_n\}_{n=1}^N$ is not equal to the span of any $N$ rank one projections.
\end{problem}

Finally, in light of Corollary \ref{com} it is natural to ask:

\begin{problem}
For the complex case, is the minimal number of projections of arbitrary rank needed for 
 phase retrieval less than or equal to the minimal number
of rank one projections needed for phase retrieval?
\end{problem}


\begin{thebibliography}{WW}

\bibitem{ABFM} B. Alexeev, A. S. Bandeira, M. Fickus, D. G. Mixon, {\it Phase retrieval with polarization}, Available online: arXiv:1210.7752

\bibitem{BE}  C. Bachoc and M. Ehler, {\em Signal reconstruction from the magnitude of subspace components}, Available online: arXiv:1209.5986.

\bibitem{BBCE} R. Balan, B. G. Bodmann, P. G. Casazza, D. Edidin, {\it Painless reconstruction from magnitudes of frame coefficients}, J. Fourier Anal. Appl. {\bf 15} (2009) 488�501.

\bibitem{BCMN}  A.S. Bandeira, J. Cahill, D.G. Mixon, and A.A. Nelson, {\em Saving phase: Injectivity and stability for phase retrieval}, Available online: arXiv:1302.4618v1.

\bibitem{BM} R. H. Bates and D. Mnyama. {\it The status of practical Fourier phase retrieval}, in W. H. Hawkes, ed., Advances in Electronics and Electron Physics, 67:1�64, 1986.

\bibitem{BCE}  R. Balan, P.G. Casazza and D. Edidin, {\em On Signal Reconstruction without Noisy Phase}, Applied and Computational
Harmonic Analysis, {\bf 20} (2006) pp. 345-356.

\bibitem{BR} C. Becchetti and L. P. Ricotti. {\it Speech recognition theory and C++ implementation}. Wiley (1999).

\bibitem{BH} B.G. Bodmann and N. Hammen. {\it Stable phase retrieval with low-redundancy frames}, Available online: arXiv:1302.5487.

\bibitem{CESV} E. J. Cand\`es, Y. Eldar, T. Strohmer, V. Voroninski, {\it Phase retrieval via matrix completion}, Available online: arXiv:1109.0573.

\bibitem{CL} E. J. Cand\`es, X. Li, {\it Solving quadratic equations via PhaseLift when there are about as many equations as unknowns}, Available online: arXiv:1208.6247.

\bibitem{CSV} E.J. Cand\`es, T. Strohmer, V. Voroninski, {\em PhaseLift: Exact and stable signal recovery from magnitude measurements via convex programming}, Available online: arXiv:1109.4499.

\bibitem{CK}  P.G. Casazza and G. Kutyniok, Eds. {\em Finite Frames:  Theory and Applications}, Birkh\"{a}user,
2013.

\bibitem{Ch}  O. Christensen, 
{\em An introduction to frames and Riesz bases}, Birkhauser, Boston (2003).

\bibitem{DH} L. Demanet, P. Hand, {\it Stable optimizationless recovery from phaseless linear measurements}, Available online: arXiv:1208.1803.

\bibitem{D} J. Drenth, {\it Principles of protein x-ray crystallography}, Springer, 2010.

\bibitem{EM} Y. C. {\it Eldar, S. Mendelson, Phase retrieval: Stability and recovery guarantees}, Available online: arXiv:1211.0872

\bibitem{F78} J. R. Fienup. {\it Reconstruction of an object from the modulus of its fourier transform}, Optics Letters, {\bf 3} (1978), 27�29.

\bibitem{F82} J. R. Fienup. {\it Phase retrieval algorithms: A comparison}, Applied Optics, {\bf 21} (15) (1982), 2758�2768.

\bibitem{MV} D. Mondragon and V. Voroninski.  {\it Determination of all pure quantum states from a minimal number of observables}, Available online: arXiv:1306.1214.

\bibitem{RJ} L. Rabiner and B. H. Juang. {\it Fundamentals of speech recognition}. Prentice Hall Signal Processing Series
(1993).

\bibitem{RBSC} J. M. Renes, R. Blume-Kohout, A. J. Scott, and C. M. Caves, {\it Symmetric Informationally Complete Quantum Measurements}, J. Math. Phys., {\bf 45}, pp. 2171�2180, 2004.

\bibitem{PDH} J. G. Proakis, J. R. Deller and J. H. L. Hansen. {\it Discrete-Time processing of speech signals}. IEEE Press
(2000).

\end{thebibliography}
\end{document}